\documentclass{amsart}
\usepackage{texfrag}
\pdfoutput=1
\usepackage{header}
\usepackage{macros}
\usepackage{feedbackdiagrams}
\usepackage{opendiagrams}
\usepackage{hiragana}
\usepackage{resizegather}
\allowdisplaybreaks

\def\titlerunning{Open Diagrams via Coend Calculus}
\def\authorrunning{Mario Román}
\author{Mario Román}
\address{Tallinn University of Technology\\Tallinn, Estonia}
\email{mroman@ttu.ee}
\thanks{Mario Román was supported by the European Union through the ESF funded Estonian IT Academy research measure (project 2014-2020.4.05.19-0001).}
\date{}
\title{Open Diagrams via Coend Calculus}
\begin{document}

\maketitle
\begin{abstract}
Morphisms in a monoidal category are usually interpreted as \emph{processes},
and graphically depicted as square boxes. In practice, we are faced with the
problem of interpreting what \emph{non-square boxes} ought to represent in
terms of the monoidal category and, more importantly, how should they be
composed. Examples of this situation include \emph{lenses} or \emph{learners}.
We propose a description of these non-square boxes, which we call \emph{open}
\emph{diagrams}, using the monoidal bicategory of profunctors. A graphical coend
calculus can then be used to reason about open diagrams and their
compositions.
\end{abstract}

\section{Introduction}
\label{sec:orgdd7751b}
\label{sec_introduction}

\subsection{Open Diagrams}
\label{sec:orgb9d323b}
\label{sec:opendiagrams}

Morphisms in monoidal categories are interpreted as processes with inputs and
outputs and generally represented by square boxes. This interpretation, however,
raises the question of how to represent a process that does not consume all the
inputs at the same time or a process that does not produce all the outputs at
the same time. For instance, consider a process that consumes an input, produces
an output, then consumes a second input and ends producing an output.
Graphically, we have a clear idea of how this process should be represented,
even if it is not a morphism in the category.

\begin{figure}[H]
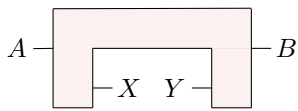

\centering
\theOpticIntro
\captionsetup{width=.8\linewidth}
\caption{A process with a non-standard shape. The input $A$ is taken at the
  beginning, then the output $X$ is produced, strictly after that, the input $Y$
  is taken; finally, the output $B$ is produced.}
\label{fig:opticintro}
\end{figure}

Reasoning graphically, it seems clear, for instance, that we should be able to
\emph{plug} a morphism connecting the first output to the second input inside
this process and get back an actual morphism of the category.

\begin{figure}[H]
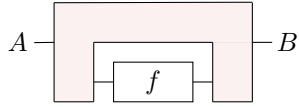

\centering
\completedOptic
\captionsetup{width=.8\linewidth}
\caption{It is possible to plug a morphism $f \colon X \to Y$ inside the
  previous process (Figure \ref{fig:opticintro}), and, importantly, get back a
  morphism $A \to B$.}
\label{fig:plugging}
\end{figure}

The particular shape depicted above has been studied by \cite{riley18}
under the name of (monoidal) \emph{optic}; it can be also called a \emph{monoidal lens}; and
it has applications in bidirectional data accessing
\cite{pickering17,boisseau18,kmett15} or compositional game theory \cite{ghani18}. A
multi-legged generalization has appeared also in quantum circuit design
\cite{chiribella08} and quantum causality \cite{uijlen17} as a notational
convention, see \cite{roman:combfeedback}. It can be shown that boxes of that
particular shape should correspond to elements of a suitable \emph{coend} (Figure
\ref{fig:coendoptic}, see also \S \ref{sec:coends} and \cite{milewski17,riley18}). The
intuition for this coend representation is to first consider a tuple of
morphisms, and then quotient out by the equivalence relation generated by
sliding morphisms along connected wires.

\begin{figure}[H]
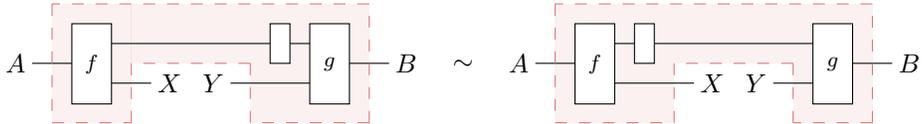

\centering
$\openOpticLeft\quad \sim \quad\openOpticRight$
\captionsetup{width=.8\linewidth}
\caption{A box of this shape is meant to represent a pair of morphisms in a
  monoidal category quotiented out by "sliding a morphism" over the upper wire.}
\label{fig:coendoptic}
\end{figure}

It has remained unclear, however, how this process should be carried in full
generality and if it was on solid ground. Are we being formal when we use these
\emph{open} or \emph{incomplete} diagrams? What happens with all the other
possible shapes that one would want to consider in a monoidal category? In
general, we cannot assume that they are squares. For instance, the second of the shapes in
Figure \ref{fig:manyshapes} has three inputs and two outputs, but the first
input cannot affect the last output; and the last input cannot affect the first
output.\footnote{This particular shape comes from a question by Nathaniel Virgo.}

\begin{figure}[H]
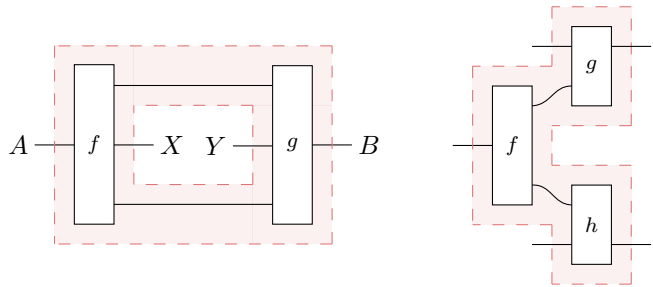

\centering
\donutdiagram \qquad \nathanieldiagram
\captionsetup{width=.8\linewidth}
\caption{Some other shapes for boxes in a monoidal category.}
\label{fig:manyshapes}
\end{figure}

This article presents the idea that incomplete diagrams should be interpreted as
valid diagrams in the monoidal bicategory of profunctors; and that compositions
of incomplete diagrams correspond to reductions that employ the monoidal
bicategory structure. At the same time, this amounts to a graphical presentation
of \emph{coend calculus}.

\subsection{Coend calculus}
\label{sec:org20d435e}
\label{sec:coends}
Coends are particular cases of colimits and \emph{coend calculus} is a practical
formalism that uses Yoneda reductions to describe isomorphisms between them.
Their dual counterparts are \emph{ends}, and formalisms for both interact nicely
in a \emph{(Co)End calculus} \cite{loregian19}.

\begin{definition}
The \textbf{coend} \(\int\nolimits^{X \in \C} P(X,X)\) of a profunctor
\(P \colon \C^{op} \times \C \to \Set\) is the coequalizer of the action of morphisms on
both arguments of the profunctor.
\[ \int^{X \in \C} P(X,X) \cong \mathrm{coeq}
  \left( \begin{tikzcd} \bigsqcup_{f \colon B \to A} P(A,B) \rar[yshift=-0.5ex, swap] \rar[yshift=0.5ex] &
  \bigsqcup_{X \in \C} P(X,X)
  \end{tikzcd}\right).
\]
An element of the coend is an equivalence class of pairs \([X, x \in P(X,X)]\) under the
equivalence relation generated by \([X, P(f,\id_X)(p)] \sim [Y, P(\id_Y,f)(p)]\) for each
\(f : X \to Y\).
\end{definition}

Our main idea is to use these equivalence relations to deal with the
quotienting arising in non-square monoidal boxes.

\begin{figure}[H]
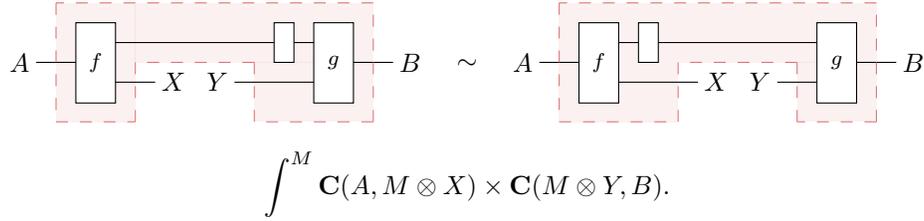

\centering
$\openOpticLeft\quad \sim\quad \openOpticRight$

\vspace{1em}

$\displaystyle \int^M \C(A, M \otimes X) \times \C(M \otimes Y, B).$
\captionsetup{width=.8\linewidth}
\caption{We can go back to Figure \ref{fig:coendoptic} to
  check how it coincides with the quotienting arising from a
  coend.}
\label{fig:coendopticwithactualcoend}
\end{figure}

\subsection{Contributions}
\label{sec:org3f42f3e}
Our first contribution is a graphical calculus of \emph{shapes} of open diagrams (\S
\ref{sec:shapes}), with semantics on the monoidal bicategory of profunctors, and
with an emphasis on representing monoidal structures. We show how to compose and
simplify shapes (\S \ref{sec:composing}). Our second contribution is a graphical
calculus of \emph{open diagrams,} in terms of the category of pointed profunctors,
and hinting at a pseudofunctorial analogue of \emph{functor boxes} \cite{mellies06} (\S
\ref{sec:opendiagrams}).

As examples, we recast the multiple ways of composing \emph{monoidal lenses} and other coend
constructions on the literature on optics (\S \ref{sec:lenses}). We also study categories with
feedback (\S \ref{sec:feedback}).

\newpage

\section{Shapes of Open Diagrams}
\label{sec:orgd92570d}
\label{sec:shapes}
In the same sense that morphisms sharing the same domain and codomain are
collected into a hom-set; open diagrams sharing the same \emph{shape} will be also collected
into a set. Our first step is a graphical language for shapes and a
compositional interpretation that assigns a set to each shape (which we
anticipate in Figure \ref{fig:coend:shapes}).

\begin{figure}[H]
  \centering
  \begin{gather*}
  \scalebox{0.85}{\donutAsShape} \quad \cong \quad \int^{M,N} \C(A, M \otimes X \otimes N) \times \C(M \otimes Y \otimes N, B), \\
  \scalebox{0.85}{\nathanielexample} \quad \cong \quad \int^{M,N} \C(I_0,M \otimes N) \times \C(I_1 \otimes M, O_1) \times \C(N \otimes I_2 , O_2).
  \end{gather*}
  \captionsetup{width=0.9\linewidth}
  \caption{The shapes of Figure \ref{fig:manyshapes}, abstracted as string diagrams, define sets.}
  \label{fig:coend:shapes}
\end{figure}

\subsection{String Diagrams}
\label{sec:orgdedbfcb}
Shapes are closed string diagrams in \(\Prof\), the monoidal bicategory of
profunctors \cite[\S 5]{loregian19}. \emph{Wires} represent small categories \((\A,\B,\C,\dots)\);
when unlabelled, they are understood to represent some fixed category.
\emph{Diagrams} with input \(\A\) and output \(\B\) are profunctors \(\A^{op} \times \B \to \Set\).
\emph{Deformations} are natural transformations. \emph{Sequential composition} of diagrams with
matching wires composes two profunctors \(P \colon \A^{op} \times \B \to \Set\) and \(Q \colon \B^{op} \times \C \to \Set\)
into the profunctor \((P \diamond Q) \colon \A^{op} \times \C \to \Set\) given by
\[(P \diamond Q)(A,C) \coloneqq \int^{B \in \B} P(A,B) \times Q(B,C).\]

\emph{Parallel composition} of diagrams uses the cartesian product of categories and the terminal
category as unit. Laying two profunctors
\(P_1 \colon \A_1^{op} \times \B_1 \to \Set\) and \(P_2 \colon \A_2^{op} \times \B_2 \to \Set\) in parallel yields the profunctor
\((P_1 \otimes P_2) \colon (\A_1 \times \A_2)^{op} \times (\B_1 \times \B_2) \to \Set\) defined by
\[(P_1 \otimes P_2)(A_1,A_2,B_1,B_2) \coloneqq P_1(A_1,B_1) \times P_2(A_2,B_2).\]
As a consequence, closed string diagrams represent sets, as profunctors \(\mathbf{1}^{op} \times \mathbf{1} \to \Set\).

The string diagrammatic calculus for monoidal bicategories has been studied by Bartlett \cite{bartlett2014quasistrict} expanding on a strictification result by
Schommer-Pries \cite{schommerpries2011classification}. It is similar to the
graphical calculus of monoidal categories, with the caveat that deformations
correspond to invertible 2-cells instead of equalities. Henceforward, the symbols
\((\to)\) and \((\cong)\) between diagrams will denote natural transformations and
natural isomorphisms, respectively. It can be also seen as a ``sliced'' version of \emph{surface diagrams}.

\begin{definition}[Input and output ports]
Every object \(A \in \C\) determines two profunctors
\((\objectYonedaA) \coloneqq \C(A,-) \colon \mathbf{1}^{op} \times \C \to \Set\) and
\((\objectCoyoneda{A}) \coloneqq \C(-,A) \colon \C^{op} \times \mathbf{1} \to \Set\)
via its contravariant and covariant Yoneda embeddings.
\end{definition}

\begin{definition}[Junctions and forks]
Every monoidal category \((\C,\otimes,I)\) has a canonical pseudomonoid structure on the
monoidal bicategory \(\Prof\) given by \((\whiteMonoid) \coloneqq \C(- \otimes -
, -)\) and \((\whiteMonoidUnit) \coloneqq \C(I,-)\), and also a canonical
pseudocomonoid structure given by \((\whiteComonoid) \coloneqq \C(- , - \otimes -)\) and
\((\whiteComonoidUnit) \coloneqq \C(-,I)\).
\end{definition}

\begin{proposition}
By definition, \((\objectYoneda{I}) \cong (\whiteMonoidUnit)\) and \((\objectCoyoneda{I}) \cong (\whiteComonoidUnit)\);
moreover,
\[\monoidalyoneda\]
In general, Yoneda embeddings are pseudofunctorial.
\end{proposition}

\subsection{Copying and discarding}
\label{sec:orgcdd64bc}
Shapes define sets in terms of coends, making them less practical for
direct manipulation. However, shapes can be reduced to more familiar
descriptions in some particular cases. For instance, if \(\C\) is cartesian monoidal, the leftmost shape of
Figure \ref{fig:simplifying} reduces to a pair of morphisms \(\C(I_0 \times I_1,
O_1)\) and \(\C(I_0 \times I_2, O_2)\). This justifies our previous
intuition, back in Figure \ref{fig:manyshapes}, that the input \(I_1\) should not be
able to affect \(O_2\), while the input \(I_2\) should not be able to affect
\(O_1\).

\begin{figure}[H]
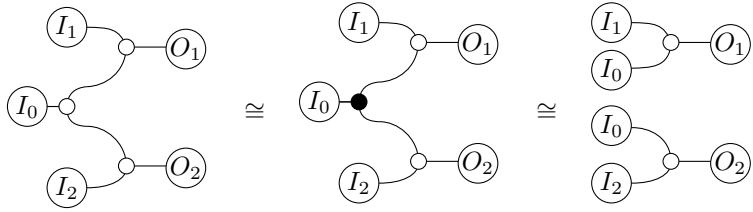

  \centering
  \begin{gather*}
  \nathanielSimplifyZero \quad \cong \quad \nathanielSimplifyOne \quad \cong \quad \nathanielSimplifyThree
  \end{gather*}
  \captionsetup{width=.8\linewidth}
  \caption{Simplifying a diagram.}
  \label{fig:simplifying}
\end{figure}

Our second step is to justify some reductions like these in the cases of
cartesian, cocartesian and symmetric monoidal categories. Every object of
the category of profunctors has already a canonical pseudocomonoid structure
lifted from \(\Cat{Cat}\) which is given by \((\blackComonoid) \coloneqq \C(-^0,-^{1})
\times \C(-^0,-^{2})\) and \((\blackComonoidUnit) \coloneqq 1\), and also a
pseudomonoid structure given by \((\blackMonoid) \coloneqq \C(-^1,-^0) \times
\C(-^2,-^0)\), and \((\blackMonoidUnit{}) \coloneqq 1\). These two structures ``copy and discard''
representable and corepresentable functors, respectively.

\begin{proposition}[Cartesian and cocartesian]
A monoidal category is cartesian if and only if
\((\whiteComonoidUnit) \cong (\blackComonoidUnit)\) and
\((\whiteComonoid) \cong (\blackComonoid)\), i.e. the monoidal structure coincides
with the canonical one. Dually, a monoidal category is cocartesian if
and only if \((\blackMonoid) \cong (\whiteMonoid)\) and
\((\whiteMonoidUnit) \cong (\blackMonoidUnit)\).
\end{proposition}

\begin{proof}
The natural isomorphism \(\C(X,Y \otimes Z) \cong \C(X,Y) \times \C(X,Z)\) is
precisely the universal property of the product; a similar reasoning holds for
initial objects, terminal objects and coproducts.
\end{proof}

\begin{proposition}[Symmetric monoidal]
If a monoidal category \(\C\) is \emph{symmetric} then its symmetric pseudomonoid
structure can be lifted from \(\Cat{Cat}\) to \(\Prof\). The braiding determines \(\sigma\colon
(\whiteTwistedMonoid) \cong (\whiteMonoid)\) and \(\sigma^\ast \colon (\whiteTwistedComonoid)\cong
(\whiteComonoid)\),
dual 2-cells in the bicategory \(\Prof\) that commute with unitors
and associators.
\end{proposition}

\subsection{Example: Lenses}
\label{sec:orgd969204}
\label{sec:lenses}

We study lenses using the graphical calculus just described.
This presents a new way of describing reductions with coend calculus
that also formalizes the intuition of lenses as \emph{diagrams with holes}.
Profunctor optics and lenses have been studied in functional
programming \cite{kmett15,milewski17,pickering17,boisseau18} for bidirectional
data accessing. The theory of optics uses coend calculus both to describe how
optics compose and how to reduce them in sufficiently well-behaved cases to
tuples of morphisms. Categories of monoidal optics and the informal interpretation of
optics as \emph{diagrams with holes} are described in \cite{riley18}.

\begin{definition}
A monoidal lens \cite[``Optic'' in Definition 2.0.1]{milewski17,pickering17,riley18}
from \(A,B \in \C\) to \(X,Y \in \C\) is an element of the following
set.
\begin{figure}[H]
\centering
$\exOptic \quad \cong \quad \displaystyle{\int^M\C(A, M \otimes X) \times \C(M \otimes Y, B)}$
\captionsetup{width=.8\linewidth}
\end{figure}
Cartesian lenses are examples of monoidal lenses that are especially
important in applications \cite{fong19,ghani18}.
\end{definition}

\begin{proposition}
In a cartesian category \(\C\), a lens
\((A,B) \to (X,Y)\) is given by a pair of morphisms \(\C(A,X)\) and
\(\C(A \times Y, B)\). In a cocartesian category, lenses are called
\emph{prisms} \cite{kmett15} and they are given by a pair of morphisms
\(\C(S, A + T)\) and \(\C(B,T)\).
\end{proposition}
\begin{proof}
We write the proof for lenses, the proof for prisms is dual and can be
obtained by mirroring the diagrams. The coend derivation can be found, for
instance, in \cite{milewski17}.
\begin{align*}
& \exOptic && \int^{M} \C(A, M \times X) \times \C(M \times Y, B) \\
\cong & \hint{(\whiteComonoid) \cong (\blackComonoid)} & \cong & \hint{\mbox{Universal property of the product}}\\
& \proofCartesianTwo && \int^{M} \C(A, M) \times \C(A,X) \times \C(M \times Y, B) \\
\cong & \hint{\mbox{Copy}} & \cong & \hint{\mbox{Yoneda lemma}} \\
& \proofCartesianThree && \C(A,X) \times \C(A \times Y, B) & \qedhere \\
\end{align*}
\end{proof}

\subsection{Example: Feedback}
\label{sec:org4d82126}
\label{sec:feedback}

Shapes do not need to be limited to a single category. For instance, we can make
use of the opposite category to introduce feedback, in the sense of the \emph{categories}
\emph{with feedback} of \cite{katis:feedback}.
Wires in the opposite category will be marked with an arrow to distinguish them.
\begin{figure}[H]
\centering
\begin{gather*}
\circDiagram{X}{Y} \quad \cong \quad \int^{M \in \C} \C(M \otimes X, M \otimes Y).
\end{gather*}
\caption{A shape with feedback, interpreted as a set.}
\end{figure}

\begin{proposition}[see {\cite{stay2013compact}}]
Profunctors form a compact closed bicategory. The dual of
a category is its opposite category.
\end{proposition}

\section{Composing and Reducing Shapes}
\label{sec:org1e2fca3}
\label{sec:composing}
We have been focusing on the invertible transformations between shapes, but
arguably the most interesting case is that of non-invertible transformations.
Our next step is to describe rules for composing and reducing diagrams that can be translated to valid coend calculus reductions.

For instance, as we saw in the
introduction (Figure \ref{fig:plugging}), a lens \((A,B) \to (X,Y)\) can be composed with a morphism \(X \to Y\) to obtain a morphism \(A \to B\).

\begin{figure}[H]
\centering
\begin{align*}
  & \exReductionOne & & \left( \int^M \C(A, M \otimes X) \times \C(M \otimes Y, B) \right) \times \C(X,Y)  \\
  \cong & \hint{\mbox{Isotopy}} & \cong  & \hint{\mbox{Continuity}} \\
  & \exReductionTwo & & \int^M \C(A, M \otimes X) \times \C(X,Y) \times \C(M \otimes Y, B)  \\
  \to & \hint{\varepsilon_X} & \to & \hint{\mbox{Composition along $X$}} \\
  & \exReductionThree & & \int^M \C(A, M \otimes Y) \times \C(M \otimes Y, B) \\
  \to & \hint{\varepsilon_Y} & \to & \hint{\mbox{Composition along $Y$}}\\
  & \exReductionFour & & \int^{M,N} \C(A, M \otimes N) \times \C(M \otimes N, B)\\
  \to & \hint{\varepsilon_{\otimes}}  & \to & \hint{\mbox{Composition along $M \otimes N$}} \\
  & \exReductionFive & & \C(A, B)
\end{align*}
\caption{Composing a lens with a morphism, formalizing Figure \ref{fig:plugging}.}
\captionsetup{width=.8\linewidth}
\end{figure}

\begin{definition}[Joining and splitting wires]
Identities and composition define natural transformations
\(\eta_{A} \colon (\ \ ) \to (\objectYoneda{A} \mkern-2mu \objectCoyoneda{A})\) and \(\varepsilon_A \colon (\objectCoyoneda{A} \objectYoneda{A}) \to (\tinyWire)\).
They determine an adjunction, as the following transformations are identities.
\[(\objectCoyoneda{A}) \overset{\eta}\to (\objectCoyoneda{A}\ \objectYoneda{A} \mkern-2mu \objectCoyoneda{A})
\overset{\varepsilon}\to (\objectCoyoneda{A}); \qquad
(\objectYoneda{A}) \overset{\varepsilon}\to (\objectYoneda{A} \mkern-2mu \objectCoyoneda{A}\ \objectYoneda{A})
\overset{\eta}\to (\objectYoneda{A}).
\]
In the same vein, junctions and forks have natural transformations
\(\varepsilon_{\otimes}\colon (\whiteMonoidBubble)\to (\tinyWire)\) and
\(\eta_\otimes \colon (\twoTinyWires)\to (\WhiteCobubble)\). They determine an
adjunction, as the following transformations are identities.
\[(\whiteMonoid) \overset{\eta}\to
(\whiteMonoid \mkern-2mu \whiteComonoid \mkern-2mu \whiteMonoid) \overset{\varepsilon}\to
(\whiteMonoid); \qquad
(\whiteComonoid) \overset{\eta}\to
(\whiteComonoid\mkern-2mu \whiteMonoid \mkern-2mu \whiteComonoid) \overset{\varepsilon}\to
(\whiteComonoid).\]
\end{definition}

\subsection{Example: Categories of Optics}
\label{sec:orgfb0badf}
\label{example:categoriesofoptics}

\begin{figure}[H]
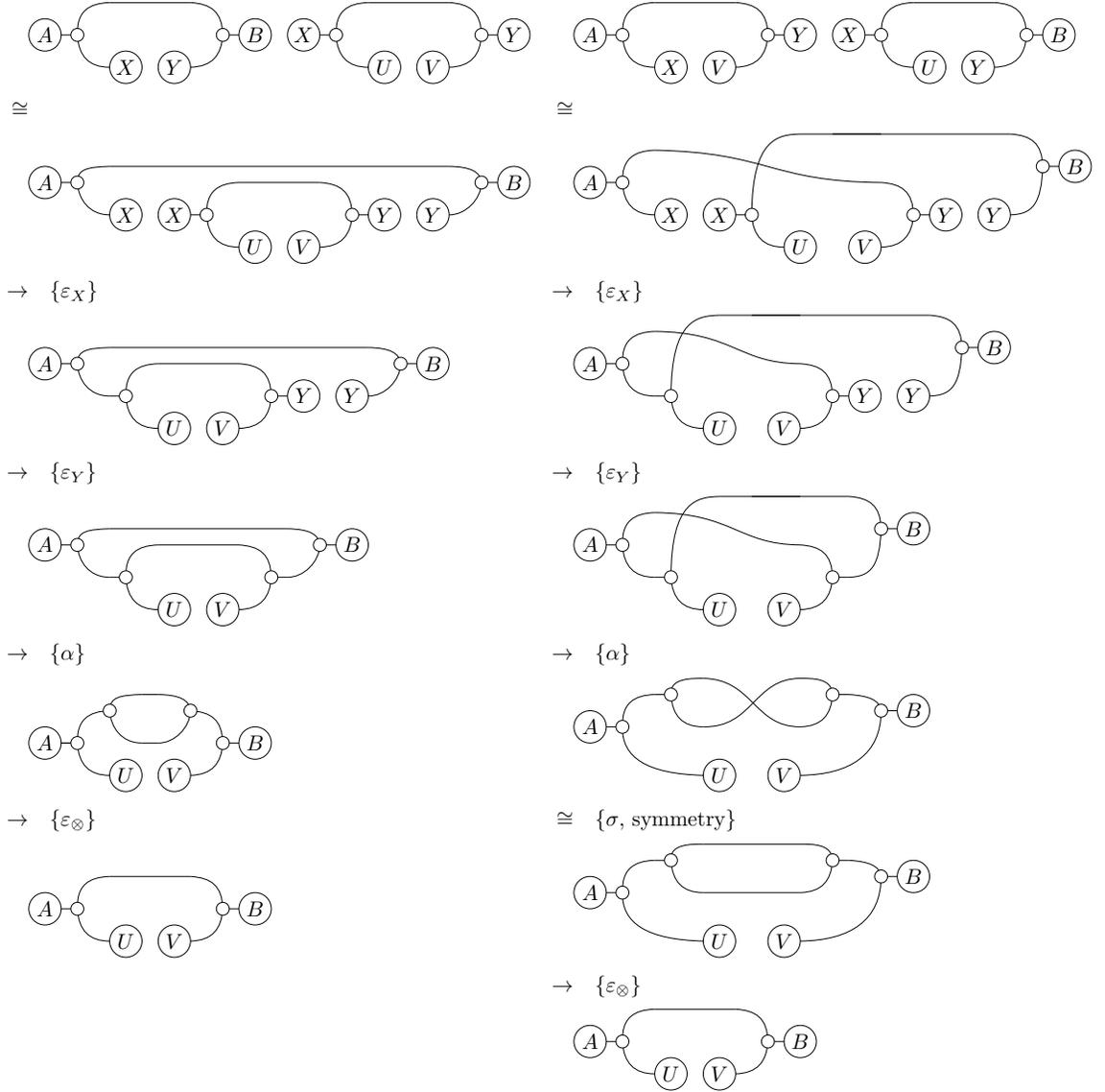

\centering
\scalebox{0.85}{\parbox{\linewidth}{
\begin{align*}
& \opticscomposeOne && \relensOne\\
\cong && \cong \\
& \opticscomposeTwo && \relensTwo\\
\to & \hint{\varepsilon_X} & \to & \hint{\varepsilon_X} \\
& \opticscomposeThree && \relensThree\\
\to & \hint{\varepsilon_Y} & \to & \hint{\varepsilon_Y} \\
& \opticscomposeFour && \relensFour\\
\to & \hint{\alpha} & \to & \hint{\alpha} \\
& \opticscomposeFive && \relensFive\\
\to & \hint{\varepsilon_{\otimes}} & \cong & \hint{\mbox{$\sigma$, symmetry}} \\
& \opticscomposeSix && \relensSix \\
&& \to & \hint{\varepsilon_{\otimes}} \\
&&& \opticscomposeSix
\end{align*}
}}
\caption{In parallel, two possible compositions of optics.}
\label{fig:twoopticcompositions}
\end{figure}
Two lenses of types \((A,B) \to (X,Y)\) and \((X,Y) \to (U,V)\) can be
composed with each other to form a \emph{category of optics} \cite{riley18} (Figure \ref{fig:twoopticcompositions}). There is,
however, another way of composing two lenses. When
the base category is symmetric, a lens \((A,Y) \to (X,V)\) can be
composed with a lens \((X,B) \to (U,Y)\) into a lens \((A,B) \to (U,Y)\). We will observe that, even if \(\Prof\)
is symmetric, the reduction explicitly uses symmetry on the base category \(\C\).

\subsection{Example: from Lenses to Dynamical Systems}
\label{sec:org2fbc70e}
In \cite[Definition 2.3.1]{schultz16}, a discrete dynamical system, a
Moore machine, is characterized to have the same data as a lens
\((A,A) \to (X,Y)\). The following derivation is a conceptual justification of
this coincidence: a lens with suitable types can be made into a morphism of the
free category with feedback \cite{katis:feedback}, subsuming particular cases
such as \emph{Moore machines}.

\begin{figure}[H]
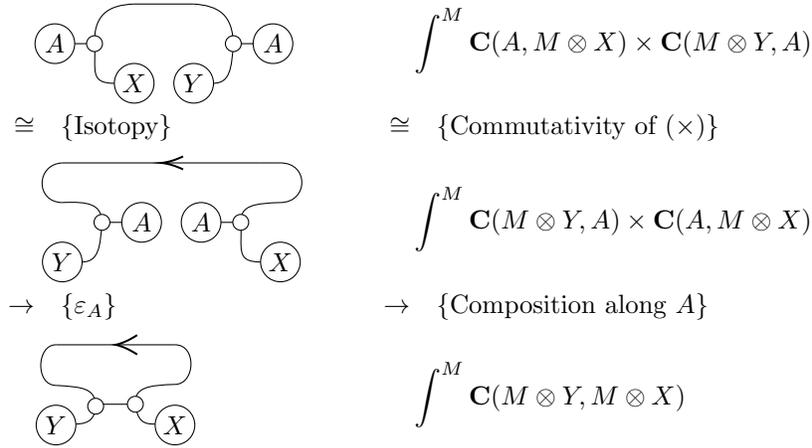

\begin{align*}
& \openLens && \int^{M} \C(A,M \otimes X) \times \C(M \otimes Y,A) \\
\cong & \hint{\mbox{Isotopy}} &\cong & \hint{\mbox{Commutativity of $(\times)$}} \\
& \openLensTwo && \int^{M} \C(M \otimes Y,A) \times \C(A,M \otimes X) \\
\to & \hint{\varepsilon_{A}} &\to & \hint{\mbox{Composition along $A$}} \\
& \circDiagram{Y}{X} && \int^{M} \C(M \otimes Y, M \otimes X)
\end{align*}
\caption{From lenses to dynamical systems.}
\end{figure}

\section{Open Diagrams}
\label{sec:org77550e8}
\label{sec:opendiagrams}

Our final contribution is to justify how to obtain the diagrams that originally
motivated this article (\emph{open diagrams}) by ``looking inside'' the
shapes. So far, the element of a set described by a shape could be only
expressed as a derivation of the shape from the empty diagram. In this section,
we show diagrams that summarize these derivations and that represent specific
elements of the shape.

\begin{figure}[H]
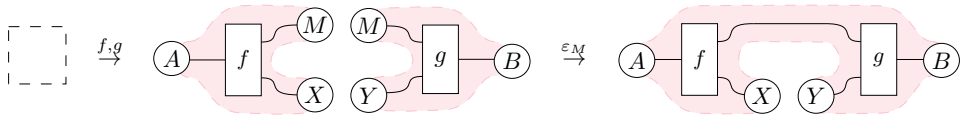

\begin{gather*}
\emptyDiagram \quad \overset{f,g}\to \quad\lensInternalTwo \quad \overset{\varepsilon_M}\to\quad \lensInternalOne
\end{gather*}
\caption{Open diagrams represent specific elements.}
\end{figure}

\subsection{Open Diagrams}
\label{sec:orga21958e}
Open diagrams will be interpreted in \(\ProfAst\), the symmetric monoidal bicategory of pointed profunctors.
Its 0-cells are categories with a chosen object; its 1-cells from \((\A,X)\) to \((\B,Y)\)
are profunctors \(P \colon \A^{op} \times \B \to \Set\) with a chosen point \(p \in P(X,Y)\); and its 2-cells
are natural transformations preserving that chosen point. The point will keep track of a
specific element of the shape.

\begin{proposition}
\label{prop:liftreductions}
Reductions on shapes can be lifted to reductions on open diagrams.
\end{proposition}
\begin{proof}
There exists a pseudofunctor \(\U \colon \ProfAst \to \Prof\) that forgets about the specific
point.  It holds that \(a \in A\) for every element \((A,a) \in \ProfAst((\mathbf{1},1),(\mathbf{1},1))\).
Natural transformations \(\alpha \colon P \to Q\) can be lifted to \(\alpha_{\ast} \colon (P,p) \to (Q,\alpha(p))\)
in a unique way, determining a discrete opfibration \(\ProfAst((\A,A),(\B,B)) \to \Prof(\A,\B)\) for every pair
of pointed categories \((\A,A)\) and \((\B,B)\).
\end{proof}

\begin{proposition}
Diagrams on the base category can be lifted to open diagrams.
\end{proposition}
\begin{proof}
Let \(\C\) be a small category. There exists a pseudofunctor \(\C \to
\ProfAst\) sending every object \(A \in \C\) to the 0-cell pair \((\C,A)\) and
every morphism \(f \in \C(A,B)\) to the 1-cell pair \((\hom_{\C}, f)\). Moreover,
when \((\C,\otimes,I)\) is monoidal, the pseudofunctor is lax and oplax monoidal
(weak pseudofunctor in \cite{moeller18}), with oplaxators being left adjoint to
laxators.

This can be called an \emph{op-ajax monoidal pseudofunctor}, following the notion
of \emph{ajax monoidal functor} from \cite{fong:regularlogic}.
\end{proof}

The graphical calculus for open diagrams can then be interpreted as the graphical
calculus of pointed profunctors enhanced with a pseudofunctorial box, in the same
vein as the functor boxes of \cite{mellies06}. Similar \emph{``internal diagrams''} have
been described before by \cite{bartlett15} as a ``graphical mnemonic notation''.

\subsection{Example: Categories of Optics}
\label{sec:org5419922}
   The lens \(\langle g,f \rangle \colon (A,B) \to (X,Y)\) is depicted as the following open diagram.
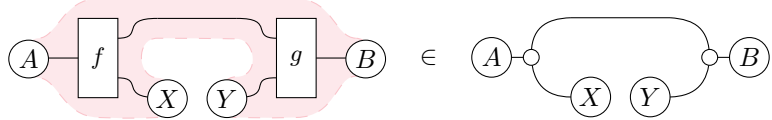
\[
\lensInternalOne \quad \in \quad \exOptic
\]
The quotienting that makes \(\langle g, (m \otimes \id_X) \circ f \rangle = \langle g \circ (m \otimes \id_Y) , f \rangle\)
is explicit in this graphical calculus. The following two diagrams are equal in
the category \(\ProfAst\): they represent the same set and the same element within it.
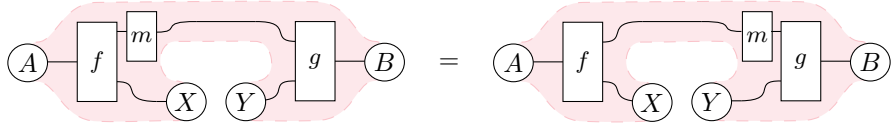
\[\lensOnTheLeft \quad = \quad \lensOnTheRight\]
Let us repeat an important caveat: the same diagram, after a deformation,
describes a different, although isomorphic, set. A diagram describes a set
only up to isomorphism. This raises a subtlety: we cannot speak of equality between open diagrams with different
shapes, for they belong to different sets. We could however speak of equality
between two open diagrams such that the shape of the first can be deformed into
the shape of the second. The deformation determines a particular isomorphism between the
sets defined by the shapes. Equality of elements on isomorphic sets is
understood to be equality after applying that isomorphism.

For instance, the following two elements, \((\lambda \circ f) \in \C(A,B \otimes I)\)
and \(f \in \C(A,B)\), are equal after the deformation given by counitality of the
pseudocomonoid structure.
\[
\left(\beforeTransport \quad \in \quad \beforeTransportSet \right) \quad
\overset{\{\lambda_\otimes\}}\cong \qquad               
\left(\afterTransport\quad \in\quad \afterTransportSet\right)\]

We will now use open diagrams to justify that both compositions
from Example \ref{example:categoriesofoptics}
determine a category. Consider two pairs of lenses of suitable types.
\[\scalebox{0.75}{\lensInternalLabel{A}{B}{X}{Y}{f_1}{g_1}\ \lensInternalLabel{X}{Y}{U}{V}{f_2}{g_2}}
\ \in\ \scalebox{0.75}{\opticLabelled{A}{X}{Y}{B}\ \opticLabelled{X}{U}{V}{Y}}\]
\[\scalebox{0.75}{\lensInternalLabel{A}{Y}{X}{V}{f_1'}{g_1'}\ \lensInternalLabel{X}{B}{U}{Y}{f_2'}{g_2'}}
\ \in\ \scalebox{0.75}{\opticLabelled{A}{X}{V}{Y}\ \opticLabelled{X}{U}{Y}{B}}\]
We can use Proposition \ref{prop:liftreductions} to lift the two compositions
in Example \ref{example:categoriesofoptics} to two deformations of open diagrams
 that send the two pairs of lenses to the following two open diagrams, respectively.
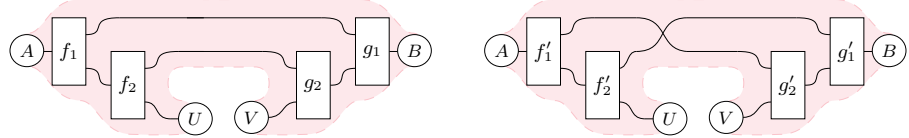
\[\scalebox{0.85}{\twoOpticsCompose \qquad \twoOpticsComposeTwisted}\]
Let us show that a category can be defined from the first composition.
Consider three lenses \(o_i\) for \(i = 1,2,3\).  We have two ways of
composing them, as \(o_1 \circ (o_2 \circ o_3)\) or
\((o_1 \circ o_2) \circ o_3\), but they both give rise to the same final
diagram, thanks to associativity of the base monoidal category.
The identity is the diagram on the right.
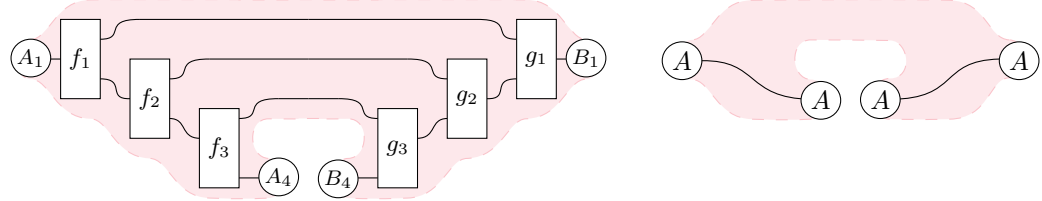
\[\tricomposedoptic \qquad \identityOpticInternal\]
For the second composition, checking associativity amounts to the following
equality. The identity is the same as in the previous case.
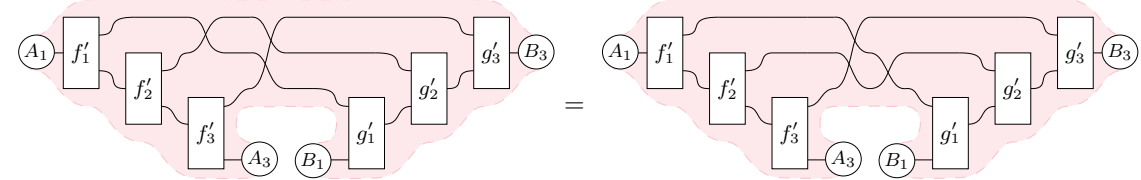
\begin{align*}
  \scalebox{0.9}{\firstAssocTwisted} =
  \scalebox{0.9}{\secondAssocTwisted}
\end{align*}
The graphical calculus is hiding at the same time the details of two structures.
The first is the quotient relation given by the coend in the monoidal bicategory
of profunctors; the second is the coherence of the base monoidal category inside
the pseudofunctorial box.

\section{Related and Further Work}
\label{sec:org9fb9af7}
The graphical calculus for profunctors can be seen as a direction in which the
graphical calculus for the cartesian bicategory of relations
\cite{bonchi:relational,fong:regularlogic} can be categorified. A notion of
\emph{cartesian bicategory} generalizing relations is discussed in
\cite{carboni2008cartesian}. For a slightly different future direction, we could try to relate
this work to many of the interesting applications of \emph{compact closed bicategories} (see \cite{stay2013compact}); such as
\emph{resistor networks}, \emph{double-entry bookeeping} \cite{katis2008partita} or
\emph{higher linear algebra} \cite{kapranov:higherlinear}.

Certain shapes open diagrams have been described in the literature. Specifically, finite
\emph{combs} were used as notation by \cite{chiribella08,uijlen17,riley18}; the relation
with lenses is described in \cite{roman:combfeedback}. Previous graphical calculi
for lenses and optics \cite{hedges17,boisseau19} have elegantly captured some
aspects of optics by working on the Kleisli or Eilenberg-Moore categories of the
Pastro-Street monoidal monad \cite{pastro08}. The present approach
diverges from previous formalisms on optics by focusing on the \emph{monoidal} structure of
the bicategory of profunctors, which seems to be crucial for the case of optics while
not considered by previous work (neither for arbitrary profunctors nor for Tambara modules).
It is more general than considering combs, as
it can express arbitrary shapes in non-symmetric monoidal categories. In any
case, it enables us to reason about categories of optics themselves; the results on optics
of \cite{profunctor20} can be greatly simplified in this calculus. We believe
that it is closer to, and it provides a formal explanation to the \emph{diagrams with
holes} of \cite[Definition 2.0.1]{riley18}, which were missing from previous
approaches.

Most of our first part can be repeated for arbitrary monoidal bicategories such
as enriched profunctors or spans. Multiple approaches to open systems (decorated
cospans \cite{fong2015decorated}, structured cospans \cite{baez2019structured})
could be related in this way to open diagrams, but we have not explored this
possibility yet. Another potential direction is to repeat this reasoning for the
case of double categories and obtain a ``tile'' version of these diagrams (see
\cite{myers2016string,hansen19}).

\section{Conclusions}
\label{sec:orgfa5d931}
We have presented a way to study and compose \emph{processes} in monoidal categories
that do not necessarily have the usual shape of a square box without losing the
benefits of the usual language of monoidal categories. Direct applications seem
to be circuit design, see \cite{chiribella08}, or the theory of optics
\cite{profunctor20}. This technique is justified by the formalism of coend
calculus \cite{loregian19} and string diagrams for monoidal bicategories
\cite{bartlett2014quasistrict}. We also argue that the graphical representation of
coend calculus is helpful to its understanding: contrasting with usual
presentations of coends that are usually centered around the Yoneda reductions,
the graphical approach seems to put more weight in the non-reversible
transformations while making most applications of Yoneda lemma transparent.
Regarding open diagrams, we can think of many other applications that have not
been described in this article: we could speak of multiple categories at the same
time and combine open diagrams of any of them using functors and adjunctions.
This work has opened many paths that we aim to further explore.

We have been working in the symmetric monoidal bicategory of profunctors for
simplicity, but similar results extend to the symmetric monoidal bicategory of
\(\mathcal{V}\mbox{-profunctors}\) for \(\mathcal{V}\) a Bénabou cosmos
\cite[\S 5]{loregian19}. We can also consider arbitrary monoidal bicategories
and drop the requirements for symmetry, copying or discarding. Finally, there is
an important shortcoming to this approach that we leave as further work: the
present graphical calculus is an extremely good tool for \emph{coend calculus}, but it
remains to see if it is so for \emph{(co)end calculus}. In other words, \emph{ends} ``enter the
picture'' only as natural transformations (see \cite{willerton:twotraces}), and
this can feel limiting even if, after applying Yoneda embeddings, it usually
suffices for most applications. As it happens with diagrammatic presentations of
regular logic \cite{bonchi:relational,fong:regularlogic}, the existential
quantifier plays a more prominent role. Diagrammatic approaches to obtaining the
universal quantifier in a situation like this go back to Peirce and are
described by \cite{nathan:firstorderpierce}.

\section*{Acknowledgements}
\label{sec:orga7c6ba1}
The author thanks seminars, discussion with, questions and/or comments by
\textsf{Paweł Sobociński}, \textsf{Edward Morehouse}, \textsf{Fosco Loregian},
\textsf{Elena Di Lavore}, \textsf{Jens Seeber}, \textsf{Jules Hedges}, \textsf{Nathaniel Virgo}, and
the whole Compositional Systems and Methods group at Tallinn University of
Technology. Mario Román was supported by the European Union through the ESF
funded Estonian IT Academy research measure (project 2014-2020.4.05.19-0001).

\bibliographystyle{eptcs}
\bibliography{bibliography}

\begin{thebibliography}{10}
\providecommand{\bibitemdeclare}[2]{}
\providecommand{\surnamestart}{}
\providecommand{\surnameend}{}
\providecommand{\urlprefix}{Available at }
\providecommand{\url}[1]{\texttt{#1}}
\providecommand{\href}[2]{\texttt{#2}}
\providecommand{\urlalt}[2]{\href{#1}{#2}}
\providecommand{\doi}[1]{doi:\urlalt{http://dx.doi.org/#1}{#1}}
\providecommand{\bibinfo}[2]{#2}

\bibitemdeclare{misc}{baez2019structured}
\bibitem{baez2019structured}
\bibinfo{author}{John~C. \surnamestart Baez\surnameend} \&
  \bibinfo{author}{Kenny \surnamestart Courser\surnameend}
  (\bibinfo{year}{2019}): \emph{\bibinfo{title}{Structured Cospans}}.

\bibitemdeclare{article}{bartlett2014quasistrict}
\bibitem{bartlett2014quasistrict}
\bibinfo{author}{Bruce \surnamestart Bartlett\surnameend}
  (\bibinfo{year}{2014}): \emph{\bibinfo{title}{Quasistrict {Symmetric}
  {Monoidal} {2-Categories} {Via} {Wire} {Diagrams}}}.
\newblock {\sl \bibinfo{journal}{arXiv preprint 1409.2148}}.

\bibitemdeclare{article}{bartlett15}
\bibitem{bartlett15}
\bibinfo{author}{Bruce \surnamestart Bartlett\surnameend},
  \bibinfo{author}{Christopher~L \surnamestart Douglas\surnameend},
  \bibinfo{author}{Christopher~J \surnamestart Schommer-Pries\surnameend} \&
  \bibinfo{author}{Jamie \surnamestart Vicary\surnameend}
  (\bibinfo{year}{2015}): \emph{\bibinfo{title}{Modular {Categories} as
  {Representations} of the {3-Dimensional} {Bordism} 2-category}}.
\newblock {\sl \bibinfo{journal}{arXiv preprint arXiv:1509.06811}}.

\bibitemdeclare{article}{boisseau19}
\bibitem{boisseau19}
\bibinfo{author}{Guillaume \surnamestart Boisseau\surnameend}
  (\bibinfo{year}{2020}): \emph{\bibinfo{title}{String {D}iagrams for
  {O}ptics}}.
\newblock {\sl \bibinfo{journal}{arXiv preprint arXiv:2002.11480}}.

\bibitemdeclare{article}{boisseau18}
\bibitem{boisseau18}
\bibinfo{author}{Guillaume \surnamestart Boisseau\surnameend} \&
  \bibinfo{author}{Jeremy \surnamestart Gibbons\surnameend}
  (\bibinfo{year}{2018}): \emph{\bibinfo{title}{{W}hat {Y}ou {N}eeda {K}now
  {A}bout {Y}oneda: {P}rofunctor {O}ptics and the {Y}oneda {L}emma
  ({F}unctional {P}earl)}}.
\newblock {\sl \bibinfo{journal}{{PACMPL}}}
  \bibinfo{volume}{2}(\bibinfo{number}{{ICFP}}), pp.
  \bibinfo{pages}{84:1--84:27}, \doi{10.1145/3236779}.

\bibitemdeclare{article}{bonchi:relational}
\bibitem{bonchi:relational}
\bibinfo{author}{Filippo \surnamestart Bonchi\surnameend},
  \bibinfo{author}{Dusko \surnamestart Pavlovic\surnameend} \&
  \bibinfo{author}{Pawe{\l} \surnamestart Soboci{\'{n}}ski\surnameend}
  (\bibinfo{year}{2017}): \emph{\bibinfo{title}{Functorial Semantics for
  Relational Theories}}.
\newblock {\sl \bibinfo{journal}{CoRR}} \bibinfo{volume}{abs/1711.08699}.
\newblock \urlprefix\url{http://arxiv.org/abs/1711.08699}.

\bibitemdeclare{book}{borceux94}
\bibitem{borceux94}
\bibinfo{author}{Francis \surnamestart Borceux\surnameend}
  (\bibinfo{year}{1994}): \emph{\bibinfo{title}{Handbook of categorical
  algebra: volume 1, Basic category theory}}.
\newblock \bibinfo{volume}{1}, \bibinfo{publisher}{Cambridge University Press}.

\bibitemdeclare{article}{carboni2008cartesian}
\bibitem{carboni2008cartesian}
\bibinfo{author}{Aurelio \surnamestart Carboni\surnameend},
  \bibinfo{author}{G~Max \surnamestart Kelly\surnameend},
  \bibinfo{author}{Robert~FC \surnamestart Walters\surnameend} \&
  \bibinfo{author}{Richard~J \surnamestart Wood\surnameend}
  (\bibinfo{year}{2008}): \emph{\bibinfo{title}{Cartesian {Bicategories}
  {II}}}.
\newblock {\sl \bibinfo{journal}{Theory and Applications of Categories}}
  \bibinfo{volume}{19}(\bibinfo{number}{6}), pp. \bibinfo{pages}{93--124}.

\bibitemdeclare{article}{chiribella08}
\bibitem{chiribella08}
\bibinfo{author}{G.~\surnamestart Chiribella\surnameend},
  \bibinfo{author}{G.~M. \surnamestart D’Ariano\surnameend} \&
  \bibinfo{author}{P.~\surnamestart Perinotti\surnameend}
  (\bibinfo{year}{2008}): \emph{\bibinfo{title}{Quantum {Circuits}
  {Architecture}}}.
\newblock {\sl \bibinfo{journal}{Physical Review Letters}}
  \bibinfo{volume}{101}(\bibinfo{number}{6}),
  \doi{10.1103/physrevlett.101.060401}.

\bibitemdeclare{article}{profunctor20}
\bibitem{profunctor20}
\bibinfo{author}{Bryce \surnamestart Clarke\surnameend}, \bibinfo{author}{Derek
  \surnamestart Elkins\surnameend}, \bibinfo{author}{Jeremy \surnamestart
  Gibbons\surnameend}, \bibinfo{author}{Fosco \surnamestart
  Loregian\surnameend}, \bibinfo{author}{Bartosz \surnamestart
  Milewski\surnameend}, \bibinfo{author}{Emily \surnamestart
  Pillmore\surnameend} \& \bibinfo{author}{Mario \surnamestart
  Román\surnameend} (\bibinfo{year}{2020}): \emph{\bibinfo{title}{Profunctor
  optics, a categorical update}}.
\newblock {\sl \bibinfo{journal}{arXiv preprint arXiv:1501.02503}}.

\bibitemdeclare{book}{coecke:picturing}
\bibitem{coecke:picturing}
\bibinfo{author}{Bob \surnamestart Coecke\surnameend} \& \bibinfo{author}{Aleks
  \surnamestart Kissinger\surnameend} (\bibinfo{year}{2017}):
  \emph{\bibinfo{title}{Picturing quantum processes}}.
\newblock \bibinfo{publisher}{Cambridge University Press}.

\bibitemdeclare{misc}{fong2015decorated}
\bibitem{fong2015decorated}
\bibinfo{author}{Brendan \surnamestart Fong\surnameend} (\bibinfo{year}{2015}):
  \emph{\bibinfo{title}{Decorated {C}ospans}}.

\bibitemdeclare{article}{fong19}
\bibitem{fong19}
\bibinfo{author}{Brendan \surnamestart Fong\surnameend} \&
  \bibinfo{author}{Michael \surnamestart Johnson\surnameend}
  (\bibinfo{year}{2019}): \emph{\bibinfo{title}{{Lenses} and {Learners}}}.
\newblock {\sl \bibinfo{journal}{CoRR}} \bibinfo{volume}{abs/1903.03671}.
\newblock \urlprefix\url{http://arxiv.org/abs/1903.03671}.

\bibitemdeclare{inproceedings}{tuyeras19}
\bibitem{tuyeras19}
\bibinfo{author}{Brendan \surnamestart Fong\surnameend}, \bibinfo{author}{David
  \surnamestart Spivak\surnameend} \& \bibinfo{author}{R{\'e}my \surnamestart
  Tuy{\'e}ras\surnameend} (\bibinfo{year}{2019}):
  \emph{\bibinfo{title}{Backprop as {Functor}: A compositional perspective on
  supervised learning}}.
\newblock In: {\sl \bibinfo{booktitle}{2019 34th Annual ACM/IEEE Symposium on
  Logic in Computer Science (LICS)}}, \bibinfo{organization}{IEEE}, pp.
  \bibinfo{pages}{1--13}.

\bibitemdeclare{article}{fong:regularlogic}
\bibitem{fong:regularlogic}
\bibinfo{author}{Brendan \surnamestart Fong\surnameend} \&
  \bibinfo{author}{David~I. \surnamestart Spivak\surnameend}
  (\bibinfo{year}{2018}): \emph{\bibinfo{title}{Graphical Regular Logic}}.
\newblock {\sl \bibinfo{journal}{CoRR}} \bibinfo{volume}{abs/1812.05765}.
\newblock \urlprefix\url{http://arxiv.org/abs/1812.05765}.

\bibitemdeclare{inproceedings}{ghani18}
\bibitem{ghani18}
\bibinfo{author}{Neil \surnamestart Ghani\surnameend}, \bibinfo{author}{Jules
  \surnamestart Hedges\surnameend}, \bibinfo{author}{Viktor \surnamestart
  Winschel\surnameend} \& \bibinfo{author}{Philipp \surnamestart
  Zahn\surnameend} (\bibinfo{year}{2018}): \emph{\bibinfo{title}{Compositional
  Game Theory}}.
\newblock In: {\sl \bibinfo{booktitle}{Proceedings of the 33rd Annual
  {ACM/IEEE} Symposium on Logic in Computer Science, {LICS} 2018, Oxford, UK,
  July 09-12, 2018}}, pp. \bibinfo{pages}{472--481},
  \doi{10.1145/3209108.3209165}.

\bibitemdeclare{article}{hansen19}
\bibitem{hansen19}
\bibinfo{author}{Linde~Wester \surnamestart Hansen\surnameend} \&
  \bibinfo{author}{Michael \surnamestart Shulman\surnameend}
  (\bibinfo{year}{2019}): \emph{\bibinfo{title}{Constructing symmetric monoidal
  bicategories functorially}}.
\newblock {\sl \bibinfo{journal}{arXiv preprint arXiv:1910.09240}}.

\bibitemdeclare{article}{nathan:firstorderpierce}
\bibitem{nathan:firstorderpierce}
\bibinfo{author}{Nathan \surnamestart Haydon\surnameend} \&
  \bibinfo{author}{Paweł \surnamestart Sobociński\surnameend}
  (\bibinfo{year}{2020}): \emph{\bibinfo{title}{Compositional Diagrammatic
  First-Order Logic}}.
\newblock {\sl \bibinfo{journal}{In Peer Review}}.

\bibitemdeclare{article}{hedges17}
\bibitem{hedges17}
\bibinfo{author}{Jules \surnamestart Hedges\surnameend} (\bibinfo{year}{2017}):
  \emph{\bibinfo{title}{Coherence for lenses and open games}}.
\newblock {\sl \bibinfo{journal}{CoRR}} \bibinfo{volume}{abs/1704.02230}.
\newblock \urlprefix\url{http://arxiv.org/abs/1704.02230}.

\bibitemdeclare{inbook}{kapranov:higherlinear}
\bibitem{kapranov:higherlinear}
\bibinfo{author}{M.~M. \surnamestart Kapranov\surnameend} \&
  \bibinfo{author}{V.~A. \surnamestart Voevodsky\surnameend}
  (\bibinfo{year}{1994}): \emph{\bibinfo{title}{2-categories and Zamolodchikov
  tetrahedra equations}}, p. \bibinfo{pages}{177{\textendash}259}.
\newblock {\sl \bibinfo{series}{Proc. Sympos. Pure
  Math.}}~\bibinfo{volume}{56}, \bibinfo{publisher}{Amer. Math. Soc.,
  Providence, RI}.

\bibitemdeclare{misc}{katis2008partita}
\bibitem{katis2008partita}
\bibinfo{author}{Piergiulio \surnamestart Katis\surnameend},
  \bibinfo{author}{N.~\surnamestart Sabadini\surnameend} \&
  \bibinfo{author}{R.~F.~C. \surnamestart Walters\surnameend}
  (\bibinfo{year}{2008}): \emph{\bibinfo{title}{On partita doppia}}.

\bibitemdeclare{article}{katis:feedback}
\bibitem{katis:feedback}
\bibinfo{author}{Piergiulio \surnamestart Katis\surnameend},
  \bibinfo{author}{Nicoletta \surnamestart Sabadini\surnameend} \&
  \bibinfo{author}{Robert F.~C. \surnamestart Walters\surnameend}
  (\bibinfo{year}{2002}): \emph{\bibinfo{title}{Feedback, trace and fixed-point
  semantics}}.
\newblock {\sl \bibinfo{journal}{{ITA}}}
  \bibinfo{volume}{36}(\bibinfo{number}{2}), pp. \bibinfo{pages}{181--194},
  \doi{10.1051/ita:2002009}.

\bibitemdeclare{inproceedings}{uijlen17}
\bibitem{uijlen17}
\bibinfo{author}{Aleks \surnamestart Kissinger\surnameend} \&
  \bibinfo{author}{Sander \surnamestart Uijlen\surnameend}
  (\bibinfo{year}{2017}): \emph{\bibinfo{title}{A categorical semantics for
  causal structure}}.
\newblock In: {\sl \bibinfo{booktitle}{2017 32nd Annual ACM/IEEE Symposium on
  Logic in Computer Science (LICS)}}, \bibinfo{organization}{IEEE}, pp.
  \bibinfo{pages}{1--12}, \doi{10.1109/LICS.2017.8005095}.

\bibitemdeclare{misc}{kmett15}
\bibitem{kmett15}
\bibinfo{author}{Edward \surnamestart Kmett\surnameend}
  (\bibinfo{year}{2012--2018}): \emph{\bibinfo{title}{lens library, Version
  4.16}}.
\newblock \bibinfo{howpublished}{Hackage
  \url{https://hackage.haskell.org/package/lens-4.16}}.

\bibitemdeclare{article}{loregian19}
\bibitem{loregian19}
\bibinfo{author}{Fosco \surnamestart Loregian\surnameend}
  (\bibinfo{year}{2019}): \emph{\bibinfo{title}{Coend calculus}}.
\newblock {\sl \bibinfo{journal}{arXiv preprint arXiv:1501.02503}}.

\bibitemdeclare{inproceedings}{mellies06}
\bibitem{mellies06}
\bibinfo{author}{Paul{-}Andr{\'{e}} \surnamestart Melli{\`{e}}s\surnameend}
  (\bibinfo{year}{2006}): \emph{\bibinfo{title}{Functorial Boxes in String
  Diagrams}}.
\newblock In: {\sl \bibinfo{booktitle}{Computer Science Logic, 20th
  International Workshop, {CSL} 2006, 15th Annual Conference of the EACSL,
  Szeged, Hungary, September 25-29, 2006, Proceedings}}, pp.
  \bibinfo{pages}{1--30}, \doi{10.1007/118746831}.

\bibitemdeclare{misc}{milewski17}
\bibitem{milewski17}
\bibinfo{author}{Bartosz \surnamestart Milewski\surnameend}
  (\bibinfo{year}{2017}): \emph{\bibinfo{title}{Profunctor optics: the
  categorical view}}.
\newblock
  \bibinfo{howpublished}{\url{https://bartoszmilewski.com/2017/07/07/profunctor-optics-the-categorical-view/}}.

\bibitemdeclare{article}{moeller18}
\bibitem{moeller18}
\bibinfo{author}{Joe \surnamestart Moeller\surnameend} \&
  \bibinfo{author}{Christina \surnamestart Vasilakopoulou\surnameend}
  (\bibinfo{year}{2018}): \emph{\bibinfo{title}{Monoidal grothendieck
  construction}}.
\newblock {\sl \bibinfo{journal}{arXiv preprint arXiv:1809.00727}}.

\bibitemdeclare{article}{myers2016string}
\bibitem{myers2016string}
\bibinfo{author}{David~Jaz \surnamestart Myers\surnameend}
  (\bibinfo{year}{2016}): \emph{\bibinfo{title}{String Diagrams For Double
  Categories and Equipments}}.
\newblock {\sl \bibinfo{journal}{arXiv preprint 1612.02762}}.

\bibitemdeclare{article}{pastro08}
\bibitem{pastro08}
\bibinfo{author}{Craig \surnamestart Pastro\surnameend} \&
  \bibinfo{author}{Ross \surnamestart Street\surnameend}
  (\bibinfo{year}{2008}): \emph{\bibinfo{title}{Doubles for monoidal
  categories}}.
\newblock {\sl \bibinfo{journal}{Theory and applications of categories}}
  \bibinfo{volume}{21}(\bibinfo{number}{4}), pp. \bibinfo{pages}{61--75}.

\bibitemdeclare{article}{pickering17}
\bibitem{pickering17}
\bibinfo{author}{Matthew \surnamestart Pickering\surnameend},
  \bibinfo{author}{Jeremy \surnamestart Gibbons\surnameend} \&
  \bibinfo{author}{Nicolas \surnamestart Wu\surnameend} (\bibinfo{year}{2017}):
  \emph{\bibinfo{title}{Profunctor {O}ptics: {M}odular {D}ata {A}ccessors}}.
\newblock {\sl \bibinfo{journal}{Programming Journal}}
  \bibinfo{volume}{1}(\bibinfo{number}{2}), p.~\bibinfo{pages}{7},
  \doi{10.22152/programming-journal.org/2017/1/7}.

\bibitemdeclare{article}{riley18}
\bibitem{riley18}
\bibinfo{author}{Mitchell \surnamestart Riley\surnameend}
  (\bibinfo{year}{2018}): \emph{\bibinfo{title}{Categories of {Optics}}}.
\newblock {\sl \bibinfo{journal}{arXiv preprint arXiv:1809.00738}}.

\bibitemdeclare{article}{roman:combfeedback}
\bibitem{roman:combfeedback}
\bibinfo{author}{Mario \surnamestart Román\surnameend} (\bibinfo{year}{2020}):
  \emph{\bibinfo{title}{Comb {D}iagrams for {D}iscrete-{T}ime {F}eedback}}.
\newblock {\sl \bibinfo{journal}{arXiv preprint arXiv:2003.06214}}.

\bibitemdeclare{misc}{schommerpries2011classification}
\bibitem{schommerpries2011classification}
\bibinfo{author}{Christopher~J. \surnamestart Schommer-Pries\surnameend}
  (\bibinfo{year}{2011}): \emph{\bibinfo{title}{The Classification of
  Two-Dimensional Extended Topological Field Theories}}.

\bibitemdeclare{misc}{schultz16}
\bibitem{schultz16}
\bibinfo{author}{Patrick \surnamestart Schultz\surnameend},
  \bibinfo{author}{David~I. \surnamestart Spivak\surnameend} \&
  \bibinfo{author}{Christina \surnamestart Vasilakopoulou\surnameend}
  (\bibinfo{year}{2016}): \emph{\bibinfo{title}{Dynamical {S}ystems and
  {S}heaves}}.

\bibitemdeclare{article}{stay2013compact}
\bibitem{stay2013compact}
\bibinfo{author}{Michael \surnamestart Stay\surnameend} (\bibinfo{year}{2013}):
  \emph{\bibinfo{title}{Compact Closed Bicategories}}.
\newblock {\sl \bibinfo{journal}{arXiv preprint arXiv:1301.1053}}.

\bibitemdeclare{misc}{willerton:twotraces}
\bibitem{willerton:twotraces}
\bibinfo{author}{Simon \surnamestart Willerton\surnameend}
  (\bibinfo{year}{2010}): \emph{\bibinfo{title}{{Two} 2-{Traces}}}.
\newblock \bibinfo{howpublished}{Slides from a talk, University of Sheffield,
  \url{http://www.simonwillerton.staff.shef.ac.uk/ftp/TwoTracesBeamerTalk.pdf}}.

\end{thebibliography}

\newpage

\section{Appendix}
\label{sec:orge2d8f8e}
\subsection{The Monoidal Bicategory of Profunctors}
\label{sec:org978abc1}
\label{sec:profunctors}

\begin{definition}
There exists a symmetric monoidal bicategory \(\Prof\) having as 0-cells the
(small) categories \(\A,\B,\C,\dots\); as 1-cells from \(\A\) to \(\B\), the
profunctors \(\A^{op} \times \B \to \Set\); as 2-cells, the natural
transformations; and as tensor product, the cartesian product of categories
\cite{loregian19}. Two profunctors \(P \colon \A^{op} \times \B \to \Set\) and
\(Q \colon \B^{op} \times \C \to \Set\) compose into the profunctor
\((P \diamond Q) \colon \A^{op} \times \C \to \Set\) defined by
\[(P \diamond Q)(A,C) \coloneqq \int^{B \in \B} P(A,B) \times Q(B,C).\]
The unit of composition in the category \(\A\) is the hom-profunctor \(\hom_{\A} \colon \A^{op} \times \A \to \Set\).
Unitors, \((\tinyWireFunctor\mkern-2mu\tinyFunctor{P}) \cong (\tinyFunctor{P}) \cong (\tinyFunctor{P}\mkern-2mu\tinyWireFunctor)\), are given by the Yoneda isomorphisms.
\begin{align*}
\lambda_{P,A,B} \colon \int^{A' \in \A} \hom_{\A}(A,A') \times P(A',B) \cong P(A,B) \\
\rho_{P,A,B} \colon \int^{B' \in \B} P(A,B') \times \hom_{\B}(B',B) \cong P(A,B).
\end{align*}
The associator \(\alpha \colon (\tinyFunctor{P_{1}}\mkern-2mu\tinyFunctor{P_{2}}) \diamond (\tinyFunctor{P_{3}}) \cong (\tinyFunctor{P_{1}}) \diamond (\tinyFunctor{P_2}\mkern-2mu\tinyFunctor{P_3})\) can be constructed from continuity and
associativity of the cartesian product of sets. Unitors and associator satisfy the pentagon
and triangular equations.

Let us describe the monoidal structure (following \cite{schommerpries2011classification}).
It uses the cartesian product of small categories and
the terminal category as unit. The monoidal product of two profunctors
\(P_1 \colon \A_1^{op} \times \B_1 \to \Set\) and
\(P_2 \colon \A_2^{op} \times \B_2 \to \Set\) is the profunctor
\((P_1 \otimes P_2) \colon (\A_1 \times \A_2)^{op} \times (\B_1 \times \B_2) \to \Set\)
defined by
\[(P_1 \otimes P_2)(A_1,A_2,B_1,B_2) \coloneqq P_1(A_1,B_1) \times P_2(A_2,B_2).\]
Unitality and associativity follow those on the cartesian structure of sets.
There exist natural isomorphisms
\(\phi^\otimes_{P_1,P_2,Q_1,Q_2} \colon (P_1 \diamond Q_1) \otimes (P_2 \diamond Q_2) \cong (P_1 \otimes P_2) \diamond (Q_1 \otimes Q_2)\)
and \(\phi^\otimes_{\A_1,\A_2} \colon (\hom_{\A_1} \otimes \hom_{\A_2}) \cong (\hom_{\A_1 \times \A_2})\)
given by continutity and the Fubini rule of coends that make it a pseudofunctor.
It has equivalences \(a \colon \A \times (\B \times \C) \cong (\A \times \B) \times \C\),
\(\lunit \colon \mathbf{1} \times \A \cong \A\) and \(\runit \colon \A \times \mathbf{1} \cong \A\), but also \(\beta \colon \A \times \B \cong \B \times \A\), with modifications making it a braided, sylleptic
and finally symmetric monoidal bicategory.
This symmetric monoidal bicategory can also be constructed from the symmetric
double category of profunctors, see \cite{hansen19}.

Every category has a dual, its opposite category. There are profunctors \((\A^{op} \times \A) \times \mathbf{1} \to \Set\)
and \(\mathbf{1}^{op} \times (\A^{op} \times \A)^{op} \to \Set\) given by further variations of the hom-profunctor; these are represented
by caps and cups. Profunctors circulate through the caps and cups as expected thanks to the Yoneda
lemma. See \cite{stay2013compact} for the description as a compact closed bicategory.
\end{definition}
\begin{definition}[Yoneda Embedding of Functors]
\label{definition_yonedaembedding_functors}
Let \(F \colon \C \to \D\) be a functor. It can be embedded as a profunctor
\((\tinyFunctor{F}) \colon \C^{op} \times \D \to \Set\) or as a profunctor
\((\tinycoFunctor{F}) \colon \D^{op} \times \C \to \Set\). Moreover, every
functor has an opposite, so it can also be embedded as a profunctor
\((\tinyOpfunctor)\colon (\D^{op})^{op} \times \C^{op} \to \Set\) or as a
profunctor
\((\tinyOpCoFunctor) \colon (\C^{op})^{op} \times \D^{op} \to \Set\). In
particular, \(F \dashv G\) precisely when
\((\tinyFunctor{F}) \cong (\tinycoFunctor{G})\).
\end{definition}

The suggestive shape of the boxes (from \cite{coecke:picturing}) is matched by
their semantics. Every category has a dual (namely, its opposite category) and
functors circulate as expected through the cups and the caps that represent
dualities.
\[\capslidingOne \cong \capslidingTwo \qquad;\qquad \cupslidingOne  \cong  \cupslidingTwo\]

\begin{proposition}
\label{yoneda:pseudofunctors}
Both Yoneda embeddings are strong monoidal pseudofunctors
\(\Cat{Cat} \to \Prof\), fully faithful on the 2-cells. Pseudofunctoriality gives
\((\tinyFunctor{F} \mkern-2mu \tinyFunctor{G}) \cong (\tinyWideFunctor)\) and
its counterpart. Monoidality gives the following isomorphism and its mirrored
counterpart.
\[\productfunctors\]
\end{proposition}
\begin{proposition}[Functors are Left Adjoints]
\label{prop:functorsleftadjoints}
In the category of profunctors, functors are left adjoints, in the sense that
there exist morphisms
\(\eta_F \colon (\tinyWire) \to (\tinyFunctor{F} \mkern-2mu \tinycoFunctor{F})\)
and
\(\varepsilon_F \colon (\tinycoFunctor{F} \mkern-2mu \tinyFunctor{F}) \to (\tinyWire)\)
and they verify the zig-zag identities. Moreover, every natural transformation
commutes with these dualities in the sense that the following are two
commutative squares.\footnote{The graphical calculus of the bicategory makes these
equations much clearer. We are emphasizing the monoidal bicategory structure
here only for the sake of coherence.}
\[\begin{tikzcd}
(\tinyWire) \rar{\eta_F} \dar[swap]{\eta_G} & (\tinyFunctor{F} \mkern-2mu \tinycoFunctor{F}) \dar{\alpha}
& (\tinycoFunctor{F} \mkern-2mu \tinyFunctor{G}) \dar{\alpha} \rar{\alpha} & (\tinycoFunctor{F} \mkern-2mu \tinyFunctor{F}) \dar{\varepsilon_{F}} \\
(\tinyFunctor{G} \mkern-2mu \tinycoFunctor{G}) \rar{\alpha} & (\tinyFunctor{F} \mkern-2mu \tinycoFunctor{G})
& (\tinycoFunctor{G} \mkern-2mu \tinyFunctor{G}) \rar{\varepsilon_{G}} & (\tinyWire)
\end{tikzcd}\]
A partial converse holds: a left adjoint profunctor is
representable when its codomain is Cauchy complete; see \cite{borceux94}.
\end{proposition}

\begin{proposition}
\label{prop:copydiscard}
Every object \(\A\) of the category of profunctors has
already a canonical pseudocomonoid structure lifted from \(\Cat{Cat}\) and given by
\((\blackComonoid) \coloneqq \A(-^0,-^{1}) \times \A(-^0,-^{2})\) and
\((\blackComonoidUnit) \coloneqq 1\); but also a pseudomonoid structure given by
\((\blackMonoid) \coloneqq \A(-^1,-^0) \times \A(-^2,-^0)\), and
\((\blackMonoidUnit{}) \coloneqq 1\). These structures \emph{copy} and \emph{discard}
representable and corepresentable functors, respectively; but they also laxly
copy and discard arbitrary profunctors.
\end{proposition}
\begin{proof}
This is a consequence of the fact the diagonal and discard functors \((\Delta) \colon \A \to \A \times \A\) and \((!) \colon \A \to \mathbf{1}\) copy and discard functors in \(\Cat{Cat}\).  Pseudofunctoriality of both Yoneda embeddings sends them to the profunctors we are describing in \(\Prof\).

On the other hand, arbitrary profunctors are laxly copied and discarded. For
instance, the following morphism shows that a profunctor \(P \colon
\A^{op} \times \B \to \Set\) is laxly copied. In the case of representable
profunctors, this is an isomorphism.
\[ \int^X P(A,X) \times \hom_{\A}(X,Y_1) \times \hom_{\B}(X,Y_2) \to
P(A,Y_1) \times P(A,Y_2). \qedhere
\]
\end{proof}

\subsection{The Monoidal Bicategory of Pointed Profunctors}
\label{sec:orga39b859}
\label{sec:pointedprofunctors}

\begin{definition}
A \emph{pointed category} \((\A,X)\) is a category \(\A\) equipped with a chosen
object \(X\), which can be regarded as a functor from the terminal category.
There exists a symmetric monoidal bicategory \(\ProfAst\) having as 0-cells
pairs \((\A,X)\) where \(\A\) is a (small) category and \(X \in \A\) is an
object of that category; 1-cells from \((\A,X) \to (\B,Y)\) pairs \((P,p)\)
given by a profunctor \(P \colon \A^{op} \times \B \to \Set\) and a point \(p
\in P(X,Y)\);
2-cells from \((P,p) \to (Q,q)\) are natural transformations
\(\eta \colon P \to Q\) such that \(\eta_{X,Y}(p) = q\). Composition of 1-cells
\((P,p) \colon (\A,X) \to (\B,Y)\) and \((Q,q)\colon (\B,Y) \to (\C,Z)\) is
given by \((Q \diamond P, \left\langle q,p \right\rangle)\), where
\(\left\langle q,p \right\rangle \in (Q\diamond P)(X,Z)\) is the equivalence
class under the coend of the pair \((q,p)\).
The identity 1-cell in \((\A,X)\) is \((\hom_{\A},\id_X) \colon (\mathbf{1}, 1) \to
(\A,X)\).

\emph{Unitors} \(\lambda_{(P,p)} \colon (\hom_\A \diamond P, \left\langle
 \id_X, p \right\rangle) \to (P,p)\) and \(\runit_{(P,p)} \colon (P
 \diamond \hom_{\A}, \left\langle p, \id_X \right\rangle) \cong (P,p)\) are
given again by the Yoneda isomorphisms.
\begin{align*}
  \lambda_P \colon \int^{Z \in \A}\hom_{\A}(X,Z) \times P(Z,Y) \cong P(X,Y)\\
  \rho_P \colon \int^{Z \in \A} P(X,Z) \times \hom_{\A}(Z,Y) \cong P(X,Y)
\end{align*}
The Yoneda isomorphisms are such that \(\lambda_P \left\langle \id_X,p
\right\rangle = \id_X \circ p = p\) and \(\rho_P \left\langle p,\id_Y
\right\rangle = p \circ \id_Y = p\).
This confirms they are valid 2-cells of \(\ProfAst\).

The \emph{associator} \(\alpha_{(P,p,Q,q,R,r)} \colon ((P \diamond Q) \diamond R,
\left\langle \left\langle p,q \right\rangle, r \right\rangle) \to (P \diamond (Q
\diamond R), \left\langle p, \left\langle q,r \right\rangle \right\rangle)\) is
given by the isomorphism described by continuity and associativity of the
cartesian product.
\begin{align*}
  & \int^{V} \left(  \int^U P(X,U) \times Q(U,V) \right) \times R(V,Y) \cong
  & \int^U P(X,U) \times \left(\int^{V}  Q(U,V)  \times R(V,Y) \right).
\end{align*}
It is defined by \(\alpha (\left\langle \left\langle p,q \right\rangle , r
\right\rangle) = \left\langle p, \left\langle q,r \right\rangle \right\rangle\),
proving that it is a valid 2-cell of \(\ProfAst\). The same triangle and pentagon equations
hold as they did in \(\Prof\).

The symmetric monoidal structure also follows from that in \(\Prof\). It uses the
cartesian product of pointed categories (choosing the pair of points) and the
terminal category with its only object, \((\One,1)\). The monoidal product of
pointed profunctors is defined by \((P_1,p_1) \otimes (P_2,p_2) \coloneqq (P_1
\otimes P_2, (p_1,p_2))\). Unitality and associativity follow again from those on
the cartesian structure of sets. The same natural isomorphisms
\(\phi^\otimes_{P_1,P_2,Q_1,Q_2} \colon (P_1 \diamond Q_1) \otimes (P_2 \diamond
Q_2) \cong (P_1 \otimes P_2) \diamond (Q_1 \otimes Q_2)\) and \(\phi^\otimes_{\A_1,\A_2} \colon (\hom_{\A_1} \otimes \hom_{\A_2}) \cong (\hom_{\A_1 \times \A_2})\)
can be also shown to preserve the points. Finally, the equivalences witnessing
associativity, left and right unitality and the braiding make the category symmetric.
\end{definition}

\subsection{Pseudofunctor box}
\label{sec:org04b6615}
\label{pseudofunctorbox}
\begin{proposition}
Let \(\A\) be a small category. There exists a pseudofunctor \(\A \to
\ProfAst\) sending every object \(A \in \A\) to the 0-cell pair \((\A,A)\) and
every morphism \(f \in \hom_\A(A,B)\) to the 1-cell pair \((\hom_{\A}, f)\). Moreover,
when \((\A,\otimes,I)\) is monoidal, the pseudofunctor is lax and oplax monoidal
(weak pseudofunctor in \cite{moeller18}, with oplaxators being left adjoint to
laxators). This would be an op-ajax monoidal pseudofunctor, following the notion
of \emph{ajax monoidal functor} from \cite{fong:regularlogic}.
\end{proposition}

We only sketch the construction. The invertible 2-cells witnessing
pseudofunctoriality use the fact that the Yoneda isomorphisms (unitors and
associators) send pairs of points to their composition, coinciding with the
composition on the base category \(\A\).

The following natural transformations make the
functor lax monoidal.
  \[\left( \tinyLaxator \right) \coloneqq (\hom(-,- \otimes -), \id_{A \otimes B}) \colon (\A \times \A,(A,B)) \to (\A,A \otimes B)\]
  \[ \left( \laxatorOne \right) \coloneqq (\hom(I,-), \id_{I}) \colon (\mathbf{1},\ast) \to (\A,I) \]
The following natural transformations make the functor oplax monoidal.
  \[ \left( \laxatorTwo \right) \coloneqq (\hom(- \otimes - ,-), \id_{A \otimes B}) \colon (\A,A \otimes B) \to (\A \times \A,(A,B)) \]
  \[ \left( \oplaxatorOne \right) \coloneqq (\hom(-, I), \id_{I}) \colon (\A,I) \to (\mathbf{1}, \ast) \]
Composition and identities give the counits and units of the adjunctions. The
fact that identity is the unit for composition makes the following
transformations be 2-cells of \(\ProfAst\).
  \[\adjunctionProductOne \overset{\varepsilon_{\mu}}\to \adjunctionProductTwo \qquad \emptyDiagram \quad \overset{\eta_{u}}\to \quad \adjunctionUnitTwo\]
  \[\coadjunctionProductOne \overset{\eta_{\mu}}{\to} \coadjunctionProductTwo \qquad \coadjunctionThree \overset{\varepsilon_u}\to \coadjunctionFour
  \qedhere\]

The following morphisms follow the cups, caps, splitting and merging structure
from \(\Prof\) in \(\ProfAst\). Morphisms circulate through them as expected:
turning to morphisms in the opposite category, being copied and discarded.
\[\left( \feedbackpiece \right) \coloneqq (\hom(-,-),\id_A) \colon (\A \times \A^{op},(A,A)) \to (\mathbf{1},1),\]
\[\left( \feedbackpieceright \right)\coloneqq (\hom(-,-),\id_A) \colon (\mathbf{1},1) \to (\A \times \A^{op},(A,A)),\]
\[\left( \internalComonoid  \right)\coloneqq (\hom(-^{0},-^{1}) \times \hom(-^{0},-^{2}),(\id_A,\id_A)) \colon (\A,A) \to (\A \times \A,(A,A)),\]
\[\left( \internalMonoid    \right)
  \coloneqq (\hom(-^{1},-^{0}) \times \hom(-^{2},-^{0}),(\id_A,\id_A)) \colon (\A \times \A,(A,A)) \to (\A,A),\]
\[\left(\internalCreate \right) \coloneqq (1,\ast) \colon (\mathbf{1},1) \to (\A,A);\qquad
\left(\internalDelete \right) \coloneqq (1,\ast) \colon (\A,A) \to (\mathbf{1},1).\]

\begin{proposition}
Let \(\A\) be a category. For every \(A \in \A\), there exist 1-cells
\[(\hom_{\A}(A,-),\id_A) \colon (\mathbf{1},1) \to (\A,A)\quad\mbox{ and }\quad(\hom_{\A}(-,A), \id_A) \colon (\A,A) \to (\mathbf{1},1)\]
given by the Yoneda embeddings of \(A\) and the identity morphism. Composition and identities define an
adjunction.
\[\emptyDiagram \quad \overset{\id_{A}}\to \quad
\adjunctionIdentities \qquad\qquad \splitIdentities \quad \overset{\circ}\to
\quad \identityWireBoxed\]
\end{proposition}

\newpage

\subsection{Example: Learners}
\label{sec:orgf2118ad}
\label{sec:learn}
A \emph{learner} \cite[Definition 4.1]{tuyeras19} in a cartesian category is given by a \emph{parameters} object \(P \in \C\), an \emph{implementation} morphism \(i \colon P \times A \to B\), and \emph{update} morphism \(u \colon P \times A \times B \to P\), and a \emph{request} morphism \(r \colon P \times A \times B \to A\).  A monoidal generalization, dinatural on the parameters object, has been proposed in \cite[Definition 6.4.1]{riley18}; the following derivation shows how it particularizes into the cartesian case.
\begin{figure}[H]
\scalebox{0.95}{\parbox{\linewidth}{
\begin{align*}
& \LearnerOne && \int^{P,Q} \C(P \times A, Q \times B) \times \C(Q \times B, P \times A) \\
\cong & \hint{(\whiteComonoid) \cong (\blackComonoid)} & \cong & \hint{\mbox{Universal property of the product}} \\
& \LearnerTwo && \int^{P,Q} \C(P \times A, Q) \times \C(P \times A, B) \times \C(Q \times B, P \times A) \\
\cong & \hint{(\blackComonoid) \mbox{ copies}} & \cong & \hint{\mbox{Yoneda lemma}} \\
& \LearnerThree && \int^{P} \C(P \times A, B) \times \C(P \times A \times B, P \times A) \\
\cong & \hint{(\blackComonoid) \mbox{ copies}} & \cong & \hint{\mbox{Universal property of the product}} \\
& \LearnerFour && \int^{P} \C(P \times A, B) \times \C(P \times A \times B, A) \times \C(P \times A \times B, P)
\end{align*}}}
\caption{From monoidal to cartesian learners.}
\end{figure}

\begin{proposition}
A pair of lenses \((U,V) \to (A,A)\) and \((V,U) \to (B,B)\) define a learner.
\end{proposition}

\begin{figure}[H]
\scalebox{0.90}{\parbox{\linewidth}{
\begin{align*}
& \lensToLearnStack && \int^{M,N} \C(U,M\!\otimes\! A) \times \C(M\!\otimes\! A,V) \times \C(V, N\!\otimes\! B) \times \C(N\!\otimes\! B, U)\\
\to & \hint{\varepsilon_{U}, \varepsilon_V}  & \to & \hint{\mbox{Composition along $U$ and $V$}}\\
& \LearnerOne && \int^{P} \C(P \times A, B) \times \C(P \times A \times B, A) \times \C(P \times A \times B, P)
\end{align*}}}
\caption{From lenses to learners.}
\end{figure}
\end{document}